\documentclass[11pt]{amsart}

\usepackage{amsmath, amssymb, enumerate, amsthm, setspace, tikz-cd, hyperref}
\usepackage[shortcuts]{extdash}
\usepackage[all]{xy}
\theoremstyle{plain}
\newtheorem{theorem}{Theorem}[section]
\newtheorem{lemma}[theorem]{Lemma}

\newtheorem{corollary}[theorem]{Corollary}

\newtheorem{conjecture}[theorem]{Conjecture}

\theoremstyle{definition}
\newtheorem{definition}[theorem]{Definition}
\newtheorem{exampletemp}[theorem]{Example}
\newtheorem{remark}[theorem]{Remark}

\DeclareMathOperator{\tr}{tr}

\DeclareMathOperator{\Ext}{Ext}

\newenvironment{example}{\begin{exampletemp}}{\hfill\qed\end{exampletemp}}
\begin{document}

\title{MF traces and the Cuntz semigroup}
\author{Christopher Schafhauser}
\address{Department of Pure Mathematics, University of Waterloo, 200 University Avenue West, Waterloo, ON, Canada, N2L 3G1}
\email{cschafhauser@uwaterloo.ca}
\date{\today}

\begin{abstract}
  A trace $\tau$ on a separable C$^*$-algebra $A$ is called \emph{matricial field} (MF) if there is a trace-preserving morphism $A \rightarrow \mathcal{Q}_\omega$, where $\mathcal{Q}_\omega$ denotes the norm ultrapower of the universal UHF-algebra $\mathcal{Q}$.  In general, the trace $\tau$ induces a state on the Cuntz semigroup $\mathrm{Cu}(A)$.  We show there is always a state-preserving morphism $\mathrm{Cu}(A) \rightarrow \mathrm{Cu}(\mathcal{Q}_\omega)$.

  As an application, if $A$ is an AI-algebra and $F$ is a free group acting on $A$, then every trace on $A \rtimes_\lambda F$ is MF.  This further implies the same result when $A$ is an AH-algebra with the ideal property such that $\mathrm{K}_1(A)$ is a torsion group.  We also use this to characterize when $A \rtimes_\lambda F$ is MF (i.e.\ admits an isometric morphism into $\mathcal{Q}_\omega$) for many simple, nuclear C$^*$-algebras $A$.
\end{abstract}

\maketitle

\section{Introduction}\label{sec:Introduction}

A C$^*$-algebra $A$ is called \emph{matricial field} (MF) if there is a net of linear, *-preserving maps $\varphi_n : A \rightarrow \mathbb{M}_{k(n)}$ where
\[ \| \varphi_n(ab) - \varphi_n(a)\varphi_n(b) \| \rightarrow 0 \quad \text{and} \quad \|\varphi_n(a)\| \rightarrow \|a\| \]
for all $a, b \in A$.  The MF-algebras were introduced by Blackadar and Kirchberg in \cite{BlackadarKirchberg:NF} where it was shown such algebras admit a certain ``generalized inductive limit'' decomposition in terms of finite dimensional algebras.  When $A$ is further assumed to be simple and nuclear, a more refined version of this decomposition result is possible which has been crucial in Elliott's Classification Program.  More precisely, a separable, simple, nuclear, MF C$^*$-algebras is the closed union of an increasing sequence of nuclear, residually finite dimensional C$^*$-subalgebras (see \cite{BlackadarKirchberg:StrongNF}).

The MF property is still poorly understood and, despite having several examples of such algebras, there are very few techniques for verifying a given C$^*$-algebra if MF.  The only obvious obstruction, and possible the only obstruction, is that if $A$ is MF, then $\tilde{A} \otimes \mathcal{K}$ does not contain an infinite projection, where $\tilde{A}$ is the unitization of $A$.  Recently, Tikuisis, White, and Winter have shown in \cite{TikuisisWhiteWinter} that every separable, nuclear C$^*$-algebra which satisfies the Universal Coefficient Theorem (UCT) and admits a faithful trace is MF (an alternate proof is given in \cite{Schafhauser:TWW}).  In the non-nuclear setting, even less is known.

A tracial state (hereafter referred to as a trace) on a C$^*$-algebra $A$ is called matricial field if one can find linear, *-preserving maps $\varphi_n : A \rightarrow \mathbb{M}_{k(n)}$ which approximately preserve the multiplication and the trace in the same sense as above; note, however, the $\varphi_n$ are not required to be approximately isometric.  There are no traces which are known not to be MF, and, as with the isometric version, there are very few tools for verifying a trace is MF aside from the Tikuisis-White-Winter Theorem mentioned above.  Note that if every trace is MF, then Connes's Embedding Problem, which asks for the same type of approximations in the 2-norm, has a positive solution.  In a sense, the MF-trace problem is a uniform version of Connes's Embedding Problem.

For crossed products of C$^*$-algebras, slightly more is known, and, moreover, the existence of MF-approximations can often be detected in the dynamics.  For example, a classical result of Pimsner in \cite{Pimsner:AFEmbedding} characterizes when a crossed product $C(X) \rtimes_\lambda \mathbb{Z}$ is MF in terms of the induced action of $\mathbb{Z}$ on $X$.  Later, N.\ Brown has shown in \cite{Brown:AFEmbedding} that if $A$ is an AF-algebra, the MF-property for $A \rtimes_\lambda \mathbb{Z}$ is characterized by the induced action of $\mathbb{Z}$ on $\mathrm{K}_0(A)$.  Results of this nature for were obtained by Matui for actions of $\mathbb{Z}$ on simple, unital A$\mathbb{T}$-algebras of real rank zero and by H.\ Lin in \cite{Lin:AFEmbedding} for actions of $\mathbb{Z}$ on AH-algebras with a faithful invariant trace.

After the integers, the next easiest class of groups has turned out to be non-abelian free groups.  A deep result of Haagerup and Thorbj{\o}rnsen \cite{HaagerupThorbjornsen} shows that if $F$ is a free group, then the group algebra $\mathrm{C}^*_\lambda(F)$ is MF.  Using this result, Kerr and Nowak have shown in \cite{KerrNowak} that if $A$ is a commutative AF-algebra and $F$ is a free group acting on $A$ (including $F = \mathbb{Z}$), the MF-property of $A \rtimes_\lambda F$ is characterized in terms of the induced action of $F$ on the Gelfand spectrum of $A$.  Rainone has extended this result to possibly non-commutative AF-algebras $A$ in \cite{Rainone:AFByFree} with the characterization given in terms of the induced action of $F$ on $\mathrm{K}_0(A)$.

In Rainone's work in \cite{Rainone:AFByFree}, a new technique was used.  The MF approximations were first constructed at the level of K-theory and then lifted to the algebra using the classification of AF-algebras.  This idea was later refined and formalized by Rainone and the author in \cite{RainoneSchafhauser}.  The idea is essentially to show that if $A \rtimes_\lambda F$ is stably finite, then there is a faithful approximate morphism $\mathrm{K}_0(A \rtimes_\lambda F) \rightarrow \mathrm{K}_0(\mathbb{M}_{k(n)})$.  From here, one uses the classification of AF-algebras to lift the compositions
\[ \mathrm{K}_0(A) \rightarrow \mathrm{K}_0(A \rtimes_\lambda F) \rightarrow \mathrm{K}_0(\mathbb{M}_{k(n)}) \]
to approximate covariant representations $(\varphi_n, u_n) : (A, F) \rightarrow \mathbb{M}_{k(n)}$.  These, in turn, induce approximate morphisms $\varphi_n \rtimes_\lambda u_n : A \rtimes_\lambda F \rightarrow \mathbb{M}_{k(n)} \otimes \mathrm{C}^*_\lambda(F)$ given by $a \lambda_s \mapsto \varphi_n(a) u_n(s) \otimes \lambda_s$.  Composing these maps with MF-approximations for $\mathrm{C}^*_\lambda(F)$ produces the MF-approximations of $A \rtimes_\lambda F$.

The same technique works for traces on crossed products and with the help of deeper classification results, extends far beyond the class of AF-algebras.  The following is the main result of \cite{RainoneSchafhauser}.

\begin{theorem}[Rainone, Schafhauser]\label{thm:RainoneSchafhauser}
Suppose $A$ is an AH-algebra with real rank zero and $F$ is a free group (possibly $\mathbb{Z}$) acting on $A$.  The following are equivalent:
\begin{enumerate}
  \item $A \rtimes_\lambda F$ is MF;
  \item $A \rtimes_\lambda F$ is stably finite;
  \item $\displaystyle \mathrm{span}_\mathbb{Z} \{ x - s\cdot x : x \in \mathrm{K}_0(A), s \in F \} \cap \mathrm{K}_0^+(A) = \{0\}$.
\end{enumerate}
Moreover, every trace on $A \rtimes_\lambda F$ is MF.
\end{theorem}

The real rank zero assumption in Theorem \ref{thm:RainoneSchafhauser} is essential if one expect to relate the MF-property of the crossed product to the $\mathrm{K}_0$-group.  For instance, if $A = C[0, 1]$ and $F = \mathbb{Z}$, then $\mathrm{K}_0(A) = \mathbb{Z}$ and hence there are no non-trivial actions of $F$ on $\mathrm{K}_0(A)$.  On the other hand, there are many actions of $F$ on $A$ for which the crossed product is not MF as can be seen from Pimsner's characterization of MF crossed products in \cite{Pimsner:AFEmbedding}.  The issue here is that $A \otimes \mathcal{K}$ has no non-trivial projections and hence the $\mathrm{K}_0$-group contains very little information about the dynamics.

To work around the real rank zero assumption, we apply the above argument with the Cuntz semigroup $\mathrm{Cu}(A)$ in place of $\mathrm{K}_0(A)$.  As the Cuntz semigroup is constructed using positive elements in $A \otimes \mathcal{K}$, the Cuntz semigroup contains rich information about the dynamics even in the absence of real rank zero.  The following is the key technical result.  Let $\mathcal{Q} = \bigotimes_n \mathbb{M}_n$ denote the universal UHF-algebra and let $\mathcal{Q}_\omega$ denote the norm ultrapower of $\mathcal{Q}$ with respect to a free ultrafilter $\omega$ on the natural numbers.

\begin{theorem}\label{thm:ConcreteCuMFApproximations}
Suppose $A$ is a separable, unital C*-algebra and $\tau$ is a trace on $A$.  There is a unital $\mathrm{Cu}$-morphism $\sigma : \mathrm{Cu}(A) \rightarrow \mathrm{Cu}(\mathcal{Q}_\omega)$ such that $\mathrm{Cu}(\mathrm{tr}_{\mathcal{Q}_\omega}) \circ \sigma = \mathrm{Cu}(\tau)$.
\end{theorem}

Recall an \emph{AI-algebra} is a sequential direct limit of algebra $C[0, 1] \otimes A_n$ for finite dimensional C$^*$-algebras $A_n$.  Following the same strategy as outlined before Theorem \ref{thm:RainoneSchafhauser} above, the Ciuperca-Elliott classification of AI-algebras together with Theorem \ref{thm:ConcreteCuMFApproximations} produces the following result.

\begin{theorem}\label{thm:AICrossedProducts}
If $A$ is an AI-algebra and $F$ is a free group acting on $A$, then every trace on $A \rtimes_\lambda F$ is MF.
\end{theorem}

For actions of $\mathbb{Z}$, the above theorem can be deduced from the Tikuisis-White-Winter Theorem.  In fact, the result holds for crossed products of locally Type I algebras by free abelian groups (see Theorem \ref{thm:locallyTypeICrossedProduct}).  For arbitrary free groups $F$, the above theorem also holds assuming only that $A \otimes \mathcal{Q}$ is an AI-algebra and hence holds whenever $A$ is an AH-algebra with the ideal property and torsion $\mathrm{K}_1$-group (Corollary \ref{cor:AHCrossedProducts} below).  Using a technical refinement of Theorem \ref{thm:AICrossedProducts} (Theorem \ref{thm:ANCCW1CrossedProducts} below) together with recent progress in classification, it is in fact possible to characterize the MF property for crossed products of certain simple, nuclear C$^*$-algebras by free groups (Corollary \ref{cor:SimpleNuclearCrossedProducts} below).

At this point, we are not able to characterize when the crossed products themselves are MF.  The same techniques would work provided one could show that for every separable, stably finite C$^*$-algebra $A$, there is a faithful $\mathrm{Cu}$-morphism $\mathrm{Cu}(A) \rightarrow \mathrm{Cu}(\mathcal{Q}_\omega)$.  This is probably within reach, however we have been unable to resolve this question.  In the case of minimal actions and, in particular, in the case when $A$ is simple, more can be said since every invariant trace on $A$ is faithful (see Corollary \ref{cor:ANCCW1CrossedProducts}).

In Section 2, we recall some terminology and results from \cite{RainoneSchafhauser} on MF approximations of dynamical systems and crossed products and also recall some definitions related to Cuntz semigroups.  Theorem \ref{thm:ConcreteCuMFApproximations} is proven in Section 3 along with the analogous statement for abstract Cuntz semigroups.  The final section, Section 4, contains our main result (Theorem \ref{thm:ANCCW1CrossedProducts}) of which Theorem \ref{thm:AICrossedProducts} is a special case.  Some consequences of this result of a related nature are also discussed.

\section{Preliminaries}\label{sec:Preliminaries}

\subsection{MF Dynamics}\label{sec:MFDynamics}

In this section, we recall some facts about MF algebras, MF traces, and the dynamical versions of these concepts.  More details can be found in \cite{RainoneSchafhauser}.

Let $A$ be a C$^*$-algebra.  We say $A$ is \emph{matricial field} (MF) if for every finite set $\mathcal{F} \subseteq A$ and for every $\varepsilon > 0$, there is an integer $k \geq 1$ and a linear, *-preserving map $\varphi : A \rightarrow \mathbb{M}_k$ such that
\[ \| \varphi(ab) - \varphi(a) \varphi(b) \| < \varepsilon \quad \text{and} \quad \| \varphi(a) \| \geq \| a \| - \varepsilon \]
for all $a \in \mathcal{F}$.  Similarly, we say a trace $\tau$ on $A$ is \emph{MF} if for every finite set $\mathcal{F} \subseteq A$ and for every $\varepsilon > 0$, there is an integer $k \geq 1$ and a linear, *-preserving map $\varphi : A \rightarrow \mathbb{M}_k$ such that
\[ \| \varphi(ab) - \varphi(a) \varphi(b) \| < \varepsilon \quad \text{and} \quad | \mathrm{tr}_k(\varphi(a)) - \tau(a) | < \varepsilon \]
for all $a \in \mathcal{F}$.  If $A$ is a unital MF-algebra, then $A$ admits an MF trace.  Conversely, if $A$ admits a faithful MF trace, then $A$ is MF.

Let $\mathcal{Q} = \bigotimes_{n=1}^\infty \mathbb{M}_n$ denote the universal UHF-algebra.  For a free ultrafilter $\omega$ on the natural numbers, define
\[ \mathcal{Q}_\omega = \{ (q_n)_{n=1}^\infty \subseteq \mathcal{Q} : \sup_n \|q_n\| < \infty \} / \{ (q_n)_{n=1}^\infty \subseteq \mathcal{Q} : \lim_{n \rightarrow \omega} \| q_n \| = 0 \}. \]
We will often identify an element $q \in \mathcal{Q}_\omega$ with a bounded sequence $(q_n) \subseteq \mathcal{Q}$ representing it.  The trace on each $\mathbb{M}_n$ induces a trace $\mathrm{tr}_{\mathcal{Q}}$ on $\mathcal{Q}$ which in turn induces a trace $\mathrm{tr}_\omega$ on $\mathcal{Q}_\omega$ via $\mathrm{tr}_\omega(q) = \lim_{n \rightarrow \omega} \tr_{\mathcal{Q}}(q_n)$.

A separable C$^*$-algebra $A$ is MF if, and only if, there is a faithful *-homomorphism $\varphi : A \rightarrow \mathcal{Q}_\omega$.  Similarly, a trace $\tau$ on a separable C$^*$-algebra $A$ is MF if, and only if, there is a *-homomorphism $\varphi : A \rightarrow \mathcal{Q}_\omega$ such that $\mathrm{tr}_\omega \circ \varphi = \tau$.

A dynamical system is a triple $(A, G, \alpha)$, where $A$ is a C$^*$-algebra, $G$ is a discrete group, and $\alpha : G \rightarrow \mathrm{Aut}(G)$ is a group homomorphism.  A \emph{covariant representation} of a dynamical system $(\varphi, u) : (A, G, \alpha) \rightarrow B$ consists of a C$^*$-algebra $B$, a *-homomorphism $\varphi : A \rightarrow B$, and a group homomorphism $u : G \rightarrow U(B)$, where $U(B)$ denotes the unitary group of $B$, such that $u_s \varphi(a) u_s^* = \varphi(\alpha_s(a))$ for all $a \in A$.

A dynamical system $(A, G, \alpha)$ is called \emph{MF} if there is a covariant representation $(\varphi, u) : (A, G, \alpha) \rightarrow \mathcal{Q}_\omega$ such that $\varphi$ is faithful.  A trace $\tau$ is called $\alpha$-\emph{MF} if there is a covariant representation $(\varphi, u) : (A, G, \alpha) \rightarrow \mathcal{Q}_\omega$ such that $\mathrm{tr}_\omega \circ \varphi = \tau$.  Note that an $\alpha$-MF trace is necessarily $\alpha$-invariant in the sense that $\tau \circ \alpha_s = \tau$ for all $s \in G$.

For a group $G$, we let $\mathrm{tr}_G$ denote the trace on the reduced group algebra $\mathrm{C}^*_\lambda(G)$ determined by $\tr_G(\lambda_s) = 0$ for $s \neq 1$.  For a C$^*$-dynamical system $(A, G, \alpha)$, we let $\mathbb{E} : A \rtimes_\lambda G \rightarrow A$ denote the expectation on the reduced crossed product $A \rtimes_\lambda G$ determined by $\mathbb{E}(\lambda_s) = 0$ for $s \neq 1$.  Note that if $\tau$ is an invariant trace on $A$, then $\tau \circ \mathbb{E}$ is a trace on $A \rtimes_\lambda G$.

The following result is Theorem 3.9 in \cite{RainoneSchafhauser}.

\begin{theorem}[Rainone, Schafhauser]\label{thm:MFDynamics}
Let $(A, G, \alpha)$ be a dynamical system such that $\mathrm{C}^*_\lambda(G)$ is exact.
\begin{enumerate}
  \item If $(A, G, \alpha)$ is MF and $\mathrm{C}^*_\lambda(G)$ is MF, then $A \rtimes_\lambda G$ is MF.
  \item If $\tau$ is an $\alpha$-MF trace on $A$ and $\mathrm{tr}_G$ is MF, then $\mathbb{E} \circ \tau$ is MF.
\end{enumerate}
\end{theorem}

For our purposes, we will only be concerned with free groups.  The following result is of fundamental importance in this paper.

\begin{theorem}
If $F$ is a free group, then $\mathrm{C}^*_\lambda(F)$ is exact and MF and $\mathrm{tr}_F$ is MF.
\end{theorem}

\begin{proof}
If $F = \mathbb{Z}$, the result follows since $\mathrm{C}^*_\lambda(\mathbb{Z})$ is commutative.  Assume $F$ is non-abelian.  By Proposition 5.1.8 in \cite{BrownOzawa}, C$^*_\lambda(\mathbb{F}_r)$ is exact.  A result of Haagerup and Thorbj{\o}rnsen \cite{HaagerupThorbjornsen} shows $\mathrm{C}^*_\lambda(F)$ is MF.  Hence $\mathrm{C}^*_\lambda(F)$ admits an MF trace.  As $\mathrm{C}^*_\lambda(F)$ has a unique trace (see Corollary VII.7.6 in \cite{Davidson:C*Book}), we have that $\mathrm{tr}_F$ is MF.
\end{proof}

\subsection{The Cuntz Semigroup}\label{sec:CuPreliminaries}

We recall the definition of and a few results about Cuntz semigroups which will be needed in Section 3.  The modern definition of Cuntz semigroups (Definition \ref{defn:CuntzSemigroup} below) was introduced in \cite{CowardElliottIvanescu}.  The theory has been extensively developed in recent years.  See, for example, \cite{AntoinePereraThiel:CuRegularity} and \cite{AntoinePereraThiel:BiCu}.

\begin{definition}
A \emph{positively ordered monoid} is a monoid $M$ together with a partial order $\leq$ such that
\begin{enumerate}
  \item if $x \in M$, then $0 \leq x$, and
  \item if $x, x', y, y' \in M$, $x \leq x'$, and $y \leq y'$, then $x + y \leq x' + y'$.
\end{enumerate}
\end{definition}

\begin{definition}
Let $(M, \leq)$ be a partially ordered set.  Given an $x, y \in M$, we say $x$ is \emph{compactly contained in} $y$, written $x \ll y$, if whenever $(y_n) \subseteq M$ is an increasing sequence in $M$ such that the supremum exists and $y \leq \sup_n y_n$, there is an $n \in \mathbb{Z}_+$ such that $x \ll y_n$.

A sequence $(y_n)_n \subseteq M$ is called \emph{rapidly increasing} if $y_n \ll y_{n+1}$ for all $n \in \mathbb{Z}_+$.
\end{definition}

\begin{definition}\label{defn:CuntzSemigroup}
A \emph{Cuntz semigroup} is a positively ordered monoid $S$ such that
\begin{enumerate}
  \item every increasing sequence in $S$ has a supremum,
  \item every element of $S$ is the supremum of a rapidly increasing sequence,
  \item if $(x_n)$ and $(y_n)$ are increasing sequences in $S$, then
  \[ \sup_n (x_n + y_n) = (\sup_n x_n) + (\sup_n y_n), \]
  and
  \item if $x, x', y, y' \in S$, $x \ll x'$, and $y \ll y'$, then $x + y \ll x' + y'$.
\end{enumerate}
A $\mathrm{Cu}$-morphism $\sigma : S \rightarrow T$ between Cuntz semigroups $S$ and $T$ is a function preserving $0, +, \leq, \ll$, and $\mathrm{sup}$.
\end{definition}

We will need a few extra properties of Cuntz semigroups.  First we need a few more definitions.

\begin{definition}
Let $S$ be a Cuntz semigroup.  A \emph{basis} for $S$ is a set $B \subseteq M$ such that for every $x \in S$, there is an increasing sequence in $B$ with supremum $x$.  We call $S$ \emph{countably based} if it admits a countable basis.
\end{definition}

\begin{definition}
Let $S$ be a Cuntz semigroup.  An element in $x \in S$ is called \emph{compact} if $x \ll x$.  We say $S$ is \emph{algebraic} if the submonoid $S_c \subseteq S$ consisting of compact elements forms a basis for $S$.
\end{definition}

Note that if $(x_n)$ is an increasing sequence with compact supremum, then $(x_n)$ is eventually constant.  This shows that if $B$ is a basis for $S$, then $S_c \subseteq B$ and, in particular, if $S$ is countably based, then $S_c$ is countable.

The following result follows from Proposition 5.5.4 in \cite{AntoinePereraThiel:CuRegularity}.

\begin{theorem}[Antoine, Perera, Thiel]\label{thm:ExtendingFromCompactSubmonoid}
Suppose $S$ and $T$ are Cuntz semigroups and $S$ is algebraic.  Any ordered monoid morphism $S_c \rightarrow T$ extends uniquely to a $\mathrm{Cu}$-morphism $S \rightarrow T$.
\end{theorem}

\begin{definition}
An \emph{order unit} for a Cuntz semigroup $S$ is a compact element $e \in S$ such that for all $x \in S$, we have $x \leq \sup \{ n e : n \geq 1 \}$.  A pair $(S, e)$ where $S$ is a Cuntz semigroup and $e$ is an order unit will be called a \emph{unital Cuntz semigroup}.  We often suppress the order unit from the notation and refer to $S$ as a unital Cuntz semigroup.

A unital $\mathrm{Cu}$-morphism $\varphi : (S, e) \rightarrow (S', e')$ between unital Cuntz semigroups is a $\mathrm{Cu}$-morphism $\varphi : S \rightarrow S'$ such that $\varphi(e) = e'$.
\end{definition}

The definition of functional below is taken from \cite{ElliottRobertSantiago} where the theory of functionals on Cuntz semigroups is extensively developed.  We will not need anything more than the definition here.

\begin{definition}
A \emph{functional} of a Cuntz semigroup $S$ is a function $\mu : S \rightarrow [0, \infty]$ which preserves $0, +, \leq$, and $\mathrm{sup}$.  If $(S, e)$ is a unital Cuntz semigroup, a \emph{state} of $S$ is a functional $\mu$ on $S$ such that $\mu(e) = 1$.
\end{definition}

The most important example of Cuntz semigroups are defined as quotients of the positive cone of a stable C$^*$-algebra.  We recall the construction below.

\begin{example}
Given a C$^*$-algebra $A$ and positive elements $a, b \in A$, define $a \precsim b$ if there is a sequence $(x_n) \subseteq A$ such that $x_n^* b x_n \rightarrow a$ and define $a \sim b$ if $a \precsim b$ and $b \precsim a$.  The \emph{Cuntz semigroup of} $A$ is defined as the set $\mathrm{Cu}(A) = (A \otimes \mathcal{K})^+ / \sim$.  Write $\langle a \rangle$ for the class in $\mathrm{Cu}(A)$ represented by an element $a \in (A \otimes \mathcal{K})^+$.  The partial order on $\mathrm{Cu}(A)$ is induced by the relation $\precsim$ and the addition is defined by orthogonal condition; that is, $\langle a \rangle + \langle b \rangle = \langle a' + b' \rangle$ where $a'$ and $b'$ are such that $\langle a \rangle = \langle a' \rangle$, $\langle b \rangle = \langle b' \rangle$, and $a' b' = 0$.

The main result of \cite{CowardElliottIvanescu} shows that $\mathrm{Cu}(A)$ is a Cuntz semigroup for all C$^*$-algebras $A$ and that $A \mapsto \mathrm{Cu}(A)$ defines a functor from C$^*$-algebras to Cuntz semigroups.  For a *-homomorphism $\varphi : A \rightarrow B$, the map $\mathrm{Cu}(\varphi) : \mathrm{Cu}(A) \rightarrow \mathrm{Cu}(B)$ is defined by $\mathrm{Cu}(\varphi)(\langle a \rangle) = \langle (\varphi \otimes \mathrm{id}_{\mathcal{K}})(a) \rangle$ for all positive $a \in A \otimes \mathcal{K}$.

If $A$ is separable, then $\mathrm{Cu}(A)$ is countably based.  If $A$ has real rank zero, then $\mathrm{Cu}(A)$ is algebraic.  For a stably finite C$^*$-algebra $A$, there is a ordered monoid isomorphism $V(A) \rightarrow S_c$, $[p] \mapsto \langle p \rangle$, where $V(A)$ is the semigroup consisting of projections in $A \otimes \mathcal{K}$ up to Murray-von Neumann equivalence.  When $A$ is unital, the element $\langle 1_A \rangle$ is an order unit for $\mathrm{Cu}(A)$.  We will always view $\mathrm{Cu}(A)$ as a unital $\mathrm{Cu}$-semigroup with this order unit.  Every trace $\tau$ on a unital C$^*$-algebra $A$ induces a state $\mathrm{Cu}(\tau)$ on $\mathrm{Cu}(A)$ via the formula
\[ \mathrm{Cu}(\tau)(\langle a \rangle) = \lim_{n \rightarrow \infty} (\tau \otimes \mathrm{Tr})(a^{1/n}) \]
for all positive $a \in A \otimes \mathcal{K}$, where $\mathrm{Tr}$ is the semifinite tracial weight on $\mathcal{K}$ normalized so that $\mathrm{Tr}(p) = 1$ for any rank one projection $p \in \mathcal{K}$.
\end{example}

The connection between the Cuntz semigroup and Theorem \ref{thm:AICrossedProducts} is given by the Ciuperca-Elliott Theorem in \cite{CiupercaElliott} which shows $\mathrm{Cu}$ classifies morphism from AI-algebras into stable rank one algebras.  In fact, using this result, Robert has shown in \cite{Robert:NCCW1} that the same result holds for sequential direct limits of Elliott-Thomsen algebras with trivial $\mathrm{K}_1$-group.  The precise statement is given in \ref{thm:CiupercaElliottRobert} below.  First we recall the definition of Elliott-Thomsen algebras.

\begin{definition}
Suppose $A$ and $B$ be C$^*$-algebras and $\varphi_0, \varphi_1 : A \rightarrow B$ are *\-/homomorphisms.  Define
\[ M(A, B, \varphi_0, \varphi_1) = \{ (a, b) \in A \oplus C([0, 1], B) : \varphi_0(a) = b(0), \varphi_1(a) = b(1) \}. \]
An \emph{Elliott-Thomsen algebra} is a C$^*$-algebra of the form $M(A, B, \varphi_0, \varphi_1)$ where $A$ and $B$ are finite dimensional.

We let $\mathcal{C}$ denote the class of Elliott-Thomsen algebras and let $\mathcal{C}_0 \subseteq \mathcal{C}$ denote the subclass consisting of Elliott-Thomsen algebras  with trivial $\mathrm{K}_1$-group.  A C$^*$-algebra is an A$\mathcal{C}$-algebra (resp.\ an A$\mathcal{C}_0$-algebra) if it is a sequential direct limit of algebras in $\mathcal{C}$ (resp.\ $\mathcal{C}_0$).
\end{definition}

\begin{theorem}[Ciuperca, Elliott, Robert]\label{thm:CiupercaElliottRobert}
Let $A$ be a unital A$\mathcal{C}_0$-algebra and let $B$ be a unital C*-algebra with stable rank one.
\begin{enumerate}
  \item If $\sigma : \mathrm{Cu}(A) \rightarrow \mathrm{Cu}(B)$ is a unital $\mathrm{Cu}$-morphism, there is a unital *-homomorphism $\varphi : A \rightarrow B$ such that $\mathrm{Cu}(\varphi) = \sigma$.
  \item If $\varphi, \psi : A \rightarrow B$ are unital *-homomorphisms with $\mathrm{Cu}(\varphi) = \mathrm{Cu}(\psi)$, then there is a sequence of unitaries $(u_n) \subseteq B$ such that
      \[ \| u_n \varphi(a) u_n^* - \psi(a) \| \rightarrow 0 \]
      for all $a \in A$.
\end{enumerate}
\end{theorem}

In our setting, we will have an action $\alpha$ of a group $G$ on an A$\mathcal{C}_0$-algebra $A$.  To produce the dynamical MF-approximations needed in Theorem \ref{thm:MFDynamics}, we first produce a $\mathrm{Cu}$-morphism $\sigma : \mathrm{Cu}(A) \rightarrow \mathrm{Cu}(\mathcal{Q}_\omega)$ such that $\sigma \circ \mathrm{Cu}(\alpha_s) = \sigma$ for all $s \in G$.  The existence statement above is used to lift $\sigma$ to a *-homomorphism $\varphi : A \rightarrow \mathcal{Q}_\omega$ and the uniqueness statement is used to compare the *-homomorphisms $\varphi$ and $\varphi \circ \alpha_s$ for $s \in G$.

\section{MF Approximations on Cuntz Semigroups}\label{sec:CuMFApproximations}

In this section, we will prove Theorem \ref{thm:ConcreteCuMFApproximations} in the introduction.  We also prove the analogous statement for abstract Cuntz semigroups in Theorem \ref{thm:AbstractCuMFApproximations} since it does not take much more work and may be of independent interest.  For the applications to crossed products in Section \ref{sec:CrossedProducts}, Theorem \ref{thm:ConcreteCuMFApproximations} will suffice.  Both version of the theorem require handling the algebraic case first and then reducing the general case to the algebraic one.

\begin{lemma}\label{lem:CuEmbeddingAlgebraicCase}
If $\mu$ is a state on a countably based, algebraic, unital Cuntz semigroup $S$, then there is a unital $\mathrm{Cu}$-morphism $\sigma : S \rightarrow \mathrm{Cu}(\mathcal{Q}_\omega)$ such that $\mathrm{Cu}(\mathrm{tr}_\omega) \circ \sigma = \mu$.
\end{lemma}

\begin{proof}
Since $S$ is countably based, $S_c$ is countable, and hence $\mu(S_c) \subseteq [0, \infty]$ is countable.  Suppose $x \in S_c$.  Since $x \leq \sup \{ n \cdot e : n \geq 1 \}$ and $x$ is compact, there is an $n \geq 1$ such that $x \leq n \cdot e$ and hence $\mu(x) \leq n$.  Therefore, $\mu(S_c) \subseteq [0, \infty)$ and, in particular, $\mu(S_c)$ generates a countable subgroup of $G \subseteq \mathbb{R}$.  By Theorem 4.8 in \cite{RainoneSchafhauser}, there is a positive, unital group morphism $\sigma_0 : G \rightarrow \mathrm{K}_0(\mathcal{Q}_\omega)$ such that $\mathrm{K}_0(\tr_\omega)(\sigma_0(g)) = g$ for all $g \in G$.  Hence there is an ordered monoid morphism $S_c \rightarrow \mathrm{K}_0^+(\mathcal{Q}_\omega)$ which preserves the unit and the state.

As $\mathcal{Q}_\omega$ has stable rank one, the canonical map $V(\mathcal{Q}_\omega) \rightarrow \mathrm{K}_0^+(\mathcal{Q}_\omega)$ is an order isomorphism and hence there is an order embedding $\mathrm{K}_0^+(\mathcal{Q}_\omega) \rightarrow \mathrm{Cu}(A)$.  Composing this map with $\sigma_0$ yields a positive, unital monoid morphism $\sigma : S_c \rightarrow \mathcal{Q}_\omega$.  Since $S$ is algebraic, there is a unique extension of $\sigma$ to a $\mathrm{Cu}$-morphism $M \rightarrow \mathcal{Q}_\omega$ by Theorem \ref{thm:ExtendingFromCompactSubmonoid}.  Since $\sigma_0$ is state-preserving and unital, so is $\sigma$.
\end{proof}

The previous lemma is enough to complete the proof of Theorem \ref{thm:ConcreteCuMFApproximations}.

\begin{proof}[Proof of Theorem \ref{thm:ConcreteCuMFApproximations}]
If $\pi_\tau$ denotes the GNS-representation of $\tau$, then $\pi_\tau(A)''$ is a von Neumann algebra and hence has real rank zero.  Since $\pi_\tau(A)$ is separable, there is a separable C$^*$-algebra $B$ with real rank zero such that $\pi_\tau(A) \subseteq B \subseteq \pi_\tau(B)''$ (see Section II.8.5 of \cite{Blackadar:Encyclopedia}).  Then $\mathrm{Cu}(B)$ is algebraic and countably based so that there is a state-preserving, unital $\mathrm{Cu}$-morphism $\mathrm{Cu}(B) \rightarrow \mathrm{Cu}(\mathcal{Q}_\omega)$.  The composition
\[ \begin{tikzcd} \mathrm{Cu}(A) \arrow{r} & \mathrm{Cu}(\pi_\tau(A)) \arrow{r} & \mathrm{Cu}(B) \arrow{r} & \mathrm{Cu}(\mathcal{Q}_\omega). \end{tikzcd} \]
is the desired $\mathrm{Cu}$-morphism.
\end{proof}

We now work towards proving a version of Theorem \ref{thm:ConcreteCuMFApproximations} for abstract Cuntz semigroups.  The next two lemmas essentially provide an analogue of the argument above for abstract Cuntz semigroups.  The role of the von Neumann algebra in the abstract setting is played by the Cuntz semigroup $M_1$ introduced by Antoine, Perera, and Thiel in Example 4.14 of \cite{AntoinePereraThiel:BiCu}.  By Proposition 4.16 in \cite{AntoinePereraThiel:BiCu}, $M_1 = \mathrm{Cu}(M)$ for any $\mathrm{II}_1$-factor $M$.  As $\mathrm{II}_1$-factors have real rank zero, $M_1$ is algebraic.  The trace on $M$ induces a state on $M_1$ denoted here by $\mu_1$.  Although $M_1$ is not countably based, we will see in Lemma \ref{lem:AlgebraicSubSemigroup}, in the same spirit as the C$^*$-algebraic argument above, that any countably based Cuntz subsemigroup of $M_1$ is contained in a countably based, algebraic Cuntz subsemigroup of $M_1$.

\begin{lemma}\label{lem:CuConnesEmbedding}
If $\mu$ is a state on a countably based algebraic, unital Cuntz semigroup $S$, then there is a unit preserving $\mathrm{Cu}$-morphism $\sigma : S \rightarrow M_1$ such that $\mu_1 \circ \sigma = \mu$.
\end{lemma}

\begin{proof}
As in Example 4.14 of \cite{AntoinePereraThiel:BiCu}, define a relation $\prec_1$ on $[0, \infty]$ by $x \prec_1 y$ if, and only if, $x \leq y$ and $x < \infty$.  Then $([0, \infty], \prec_1)$ is a $\mathrm{Q}$-semigroup in the sense of Definition 4.1 in \cite{AntoinePereraThiel:BiCu}.  The pairs $(S, \ll)$ and $(M_1, \ll)$ are also $\mathrm{Q}$-semigroups and every $\mathrm{Q}$-morphism $(S, \ll) \rightarrow (M_1, \ll)$ is a $\mathrm{Cu}$-morphism.  By the construction of $M_1$ in Example 4.14 of \cite{AntoinePereraThiel:BiCu}, there is a $Q$-morphism $\sigma_0 : ([0, \infty], \prec_1) \rightarrow (M_1, \ll)$ such that $\mu_1(\sigma_0(x)) = x$ for all $x \in [0, \infty]$.

We claim $\mu : S \rightarrow [0, \infty]$ is a $\mathrm{Q}$-morphism.  By definition, $\mu$ preserves zero, addition, order, and suprema of increasing sequences.  It only remains to show $\mu$ preserves the auxiliary relation; that is, if $x \ll y$ in $S$, then $\mu(x) \prec_1 \mu(y)$ in $[0, \infty]$.  To this end, suppose $x \ll y$ in $S$.  Then $x \leq y$ and hence $\mu(x) \leq \mu(y)$.  Also, since $y \leq \sup \{ n \cdot u : n \geq 1 \}$ in $S$, there is an $n \geq 1$ such that $x \leq n \cdot u$.  Therefore, $\mu(x) \leq n < \infty$.  This shows $\mu(x) \prec_1 \mu(y)$ and that $\mu$ is a $\mathrm{Q}$-morphism.

To complete the proof, define $\sigma = \sigma_0 \circ \mu : S \rightarrow M_0$.
\end{proof}

\begin{lemma}\label{lem:AlgebraicSubSemigroup}
If $S \subseteq T$ are Cuntz semigroups such that $S$ is countably based and $T$ is algebraic, then there is a countably based, algebraic Cuntz semigroup $S'$ with $S \subseteq S' \subseteq T$.
\end{lemma}

\begin{proof}
Let $B$ denote a countable basis for $S$.  For each $x \in B$, there is a sequence $(y_{x, n})_n$ of compact elements in $T$ such that $x = \sup_n y_{x, n}$.  Let $S'_0 \subseteq N$ denote the monoid generated by $\{ y_{x,n} : x \in B, n \geq 1 \}$.  Let $S'$ denote the set of suprema of increasing sequences of elements in $S'_0$ and endow $S'$ with the order structure inherited from $T$. Note that $S'_0$ is countable and each element of $S'_0$ is a compact element of $T$.  We will show $S'$ is a countably based, algebraic Cuntz semigroup containing $S$.

If $x, y \in S'$, there are increasing sequences $(x_n)$ and $(y_n)$ in $S'_0$ such that $x = \sup x_n$ and $y = \sup y_n$.  Then $x_n + y_n \in S'_0$ and $x + y = \sup (x_n + y_n)$.  Hence $S'$ is a submonoid of $T$.  Now suppose $(x_n)$ is an increasing sequence in $S'$ and let $x = \sup x_n$, where the supremum is taken in $T$.  We claim that $x \in S'$.  For each $n \geq 1$, there is an increasing sequence $(y_{n, k})_k$ in $S'_0$ such that $x_n = \sup_k y_{n, k}$.  Since $T$ is a Cuntz semigroup, there is a rapidly increasing sequence $(x'_j)$ in $T$ with supremum $x$.  Since $x'_2 \ll x = \sup_n x_n$, there is an integer $n(1) \geq 1$ such that $x'_2 \leq x_{n(1)}$.  Then $x'_1 \ll x'_2 \leq x_{n(1)}$ and, in particular, $x'_1 \ll x_{n(1)}$.  Continuing in this fashion, we produce an increasing sequence $(n(j))_j$ of integers such that $x'_j \ll x_{n(j)}$.  As $x'_1 \ll x_{n(1)} = \sup_k y_{n(1), k}$, there is an integer $k(1)$ such that $x'_1 \leq y_{n(1), k(1)}$.  Also, $x'_2 \ll x_{n(2)} = \sup_k y_{n(2), k}$ and $y_{n(1), k(1)} \ll x_{n(1)} \leq x_{n(2)} = \sup_k y_{n(2), k}$ (since $y_{n(1), k(1)}$ is compact).  Hence there is an integer $k(2) \geq 1$ such that $x'_2 \leq y_{n(2), k(2)}$ and $y_{n(1), k(1)} \leq y_{n(2), k(2)}$.  Continue inductively and define $z_j = y_{n(j), k(j)}$.  Then $z_j$ is an increasing sequence in $S'_0$ and $x'_j \leq z_j \leq x_{n(j)}$ for every $j \geq 1$.  As $\sup x'_j = \sup x_{n(j)} = x$, it follows that $\sup_j z_j = x$ and hence $x \in S'$.

To show $S'$ is a Cuntz semigroup, it now suffices to show the relation $\ll$ computed with respect to $S'$ agrees with the relation $\ll$ computed with respect to $T$ for all pairs of elements in $S'$.  For the moment, let us denote compact containment relative to the order structure on $S'$ by $\ll'$.  Then, with this notation, we mush show that for $x, y \in S'$, we have $x \ll' y$ if and only if $x \ll y$.

For $x, y \in S'$, it is clear that $x \ll y$ implies $x \ll' y$.  Suppose conversely, $x, y \in S'$ and $x \ll' y$.  Suppose $z_n$ is an increasing sequence in $T$ such that $y \leq \sup_n z_n$.  Fix an increasing sequence $y_k$ in $S'_0$ such that $y = \sup y_k$.  As $x \ll' y$, there is an integer $k \geq 1$ such that $x \leq y_k$. As $y_k$ is compact, and $y_k \leq y \leq \sup_n z_n$, there is an integer $n \geq 1$ such that $y_k \leq z_n$.  Hence $x \leq z_n$ and it follows that $x \ll y$.  This shows the relations $\ll$ and $\ll'$ agree on $S'$ as claimed.

We have $S'$ is a Cuntz semigroup with the addition and order inherited from $T$ and the inclusion $S' \hookrightarrow T$ is a $\mathrm{Cu}$-embedding.  As $S'_0$ is countable and consists of compact elements and as every element in $S'$ is a supremum of an increasing sequence in $S'_0$, we have that $S'$ is countably based and algebraic.  Since $S'_0$ contains a basis for $S$, it follows that $S \subseteq S'$, and this completes the proof.
\end{proof}

Combining the previous three lemmas yields the following result which provides an abstract version of Theorem \ref{thm:ConcreteCuMFApproximations}.

\begin{theorem}\label{thm:AbstractCuMFApproximations}
If $S$ is a countably based, unital Cuntz semigroup and $\mu$ is a state on $S$, there is a unital $\mathrm{Cu}$-morphism $\sigma : S \rightarrow \mathrm{Cu}(\mathcal{Q}_\omega)$ such that $\mathrm{Cu}(\mathrm{tr}_\omega) \circ \sigma = \mu$.
\end{theorem}

\begin{proof}
There is a state-preserving $\mathrm{Cu}$-morphism $S \rightarrow M_1$ by Lemma \ref{lem:CuConnesEmbedding} and hence after replacing $S$ with the Cuntz semigroup generated by the image of $S$ under this morphism, we may assume $S \subseteq M_1$.  As $M_1$ is the Cuntz semigroup of a $\mathrm{II}_1$-factor and $\mathrm{II}_1$-factors have real rank zero, $M_1$ is algebraic.  By Lemma \ref{lem:AlgebraicSubSemigroup}, there is a countably based, algebraic $\mathrm{Cu}$-semigroup $S' \subseteq M_1$ containing $S$.  Now Lemma \ref{lem:CuEmbeddingAlgebraicCase} produces an state-preserving $\mathrm{Cu}$-morphism $S'$ to $\mathrm{Cu}(\mathcal{Q}_\omega)$.  Composing this morphism with the inclusion $S \hookrightarrow S'$ yields the result.
\end{proof}

\section{Applications to Crossed Products}\label{sec:CrossedProducts}

We now work to prove Theorem \ref{thm:AICrossedProducts} and some consequences of this result.  The proofs in this section are nearly identical to those used in \cite{RainoneSchafhauser} in the real rank zero setting.  The key difference is using the Ciuperca-Elliott-Robert classification (Theorem \ref{thm:CiupercaElliottRobert}) in place of the classification of real rank zero A$\mathbb{T}$-algebras and using the $\mathrm{Cu}$-morphisms provided by Theorem \ref{thm:ConcreteCuMFApproximations} in place of the analogous result for $\mathrm{K}_0$-groups.

Theorem \ref{thm:MFDynamics} provides a method for showing traces on crossed products which factor through the expectation are MF, but to handle arbitrary traces, we use the following two results.  The first is due to de la Harpe and Skandalis in \cite{HarpeSkandalis} and shows for non-abelian free groups, all traces factor through the expectation.  In the case of actions of the integers, the structure of traces is much more complicated, but in our setting, we can work around this using the Tikuisis-White-Winter Theorem (see \cite{TikuisisWhiteWinter} and \cite{Schafhauser:TWW}).

\begin{theorem}[de la Harpe, Skandalis]\label{thm:HarpeSkandalis}
If $A$ is a C*-algebra and $F$ is a non-abelian free group acting on $A$, then every trace on $A \rtimes_\lambda F$ has the form $\tau \circ \mathbb{E}$ where $\tau$ is a $F$-invariant trace on $A$ and $\mathbb{E} : A \rtimes_\lambda F \rightarrow A$ is the canonical expectation.
\end{theorem}

Recall that an algebra is \emph{locally Type I} if the following approximation property holds: for every finite set $\mathcal{F} \subseteq A$ and for every $\varepsilon > 0$, there is a Type I subalgebra $B \subseteq A$ such that for every $a \in \mathcal{F}$, there is a $b \in B$ such that $\| a - b \| < \varepsilon$.  Note that as quotients of Type I algebras are Type I, we also have the quotients of locally Type I algebras are locally Type I.

\begin{theorem}\label{thm:locallyTypeICrossedProduct}
Suppose $A$ is a locally Type I algebra and $G$ is a free abelian group.  Then every trace on $A \rtimes_\lambda G$ is quasidiagonal and, in particular, is MF.
\end{theorem}

\begin{proof}
We may assume $A$ is separable and that $G$ is finitely generated; say $G \cong \mathbb{Z}^k$.

Let $\tau$ be a trace on $A \rtimes_\lambda G$ and let $J = \{ a \in A : \tau(a^*a) = 0 \}$.  A standard application of the Cauchy-Schwarz inequality shows $J$ is a two-sided ideal in $A$.  Moreover, $J$ is easily seen to be $G$-invariant.  Now, $J \rtimes_\lambda G$ is an ideal in $A \rtimes_\lambda G$ and $\tau$ vanishes on $J \rtimes_\lambda G$.  Hence $\tau$ induces a trace $\bar{\tau}$ on $(A \rtimes_\lambda G) / (J \rtimes_\lambda G) \cong (A / J) \rtimes_\lambda G$ and $\bar{\tau}|_{A/J}$ is faithful.  To show $\tau$ is quasidiagonal, it suffices to show $\bar{\tau}$ is quasidiagonal.

As $A$ is locally Type I, so is $A / J$.  By Theorem 1.1 in \cite{Dadarlat:locallyUCT}, $A/J$ satisfies the UCT.  Since the UCT is preserved by taking crossed products by $\mathbb{Z}$, the algebra $(A / J) \rtimes_\lambda G$ satisfies the UCT.  The trace $\bar{\tau}|_{A/J} \circ \mathbb{E}$ is a faithful trace on $(A / J) \rtimes_\lambda G$ and hence the Tikuisis-White-Winter Theorem implies every trace on $(A / J) \rtimes_\lambda G$ is quasidiagonal (see Corollary 6.1 in \cite{TikuisisWhiteWinter}).
\end{proof}

The following contains Theorem \ref{thm:AICrossedProducts} as a special case.

\begin{theorem}\label{thm:ANCCW1CrossedProducts}
Suppose $A$ is a C$^*$-algebra such that $A \otimes \mathcal{Q}$ is a A$\mathcal{C}_0$-algebra.  If $F$ is a free group acting on $A$, then every trace on $A \rtimes_\lambda F$ is MF.
\end{theorem}

\begin{proof}
The case when $F = \mathbb{Z}$ follows from Theorem \ref{thm:locallyTypeICrossedProduct}.  When $F$ is a non-abelian free group, every trace on $A \rtimes_\lambda F$ has the form $\tau \circ \mathbb{E}$ for an invariant trace $\tau$ on $A$ by Theorem \ref{thm:HarpeSkandalis}.  Hence it suffices to show every invariant trace on $A$ is $\alpha$-MF where $\alpha$ is the action of $F$ on $A$.

We may assume $A$ is unital and $F$ is freely generated by finitely many elements $s_1, \ldots, s_n$.  After replacing $A$ with $A \otimes \mathcal{Q}$, we may assume $A$ is an A$\mathcal{C}_0$-algebra.  Let $\alpha_j = \alpha_{s_j}$ and $\lambda_j = \lambda_{s_j} \in A \rtimes_\lambda F$ for $j = 1, \ldots, n$.  It suffices to show there is a trace-preserving *-homomorphism $\varphi : A \rightarrow \mathcal{Q}_\omega$ and unitaries $u_1, \ldots, u_n \in \mathcal{Q}_\omega$ such that $u_j \varphi(a) u_j^* = \varphi(\alpha_j(a))$.

The trace $\tau$ induces a trace $\tilde{\tau} = \tau \circ \mathbb{E}$ on $A \rtimes_\lambda F$ where $\mathbb{E} : A \rtimes_\lambda F \rightarrow A$ is the canonical conditional expectation.  By Theorem \ref{thm:ConcreteCuMFApproximations}, there is a state-preserving $\mathrm{Cu}$-morphism
\[ \sigma : \mathrm{Cu}(A \rtimes_\lambda F) \rightarrow \mathrm{Cu}(\mathcal{Q}_\omega). \]
Let $\iota : A \rightarrow A \rtimes_\lambda F$ denote the canonical inclusion and note that since $\iota$ is trace-preserving, the composition $\sigma \circ \mathrm{Cu}(\iota)$ is state-preserving.

By the Ciuperca-Elliott-Robert Theorem (Theorem \ref{thm:CiupercaElliottRobert} above), there is a unital *-homomorphism $\varphi : A \rightarrow \mathcal{Q}_\omega$ such that $\mathrm{Cu}(\varphi) = \sigma \circ \mathrm{Cu}(\iota)$.  For $j = 1, \ldots, n$, we have
\[ \mathrm{Cu}(\varphi \circ \alpha_j) = \sigma \circ \mathrm{Cu}(\iota \circ \alpha_j) = \sigma \circ \mathrm{Cu}(\mathrm{ad}(\lambda_j) \circ \iota) = \sigma \circ \mathrm{Cu}(\iota) = \mathrm{Cu}(\varphi). \]
Applying the Ciuperca-Elliott-Robert Theorem again shows $\varphi$ and $\varphi \circ \alpha_j$ are approximately unitarily equivalent and hence are unitarily equivalent by applying a reindexing argument in $\mathcal{Q}_\omega$.  That is there are unitaries $u_j \in \mathcal{Q}_\omega$ such that $\varphi \circ \alpha_j = \mathrm{ad}(u_j) \circ \varphi$ as desired.
\end{proof}

\begin{corollary}\label{cor:ANCCW1CrossedProducts}
Suppose $A$ a unital C*-algebra such that $A \otimes \mathcal{Q}$ is an A$\mathcal{C}_0$-algebra and $F$ is a free group acting minimally on $A$.  The following are equivalent:
\begin{enumerate}
  \item $A \rtimes_\lambda F$ is MF;
  \item $A \rtimes_\lambda F$ is stably finite;
  \item $A$ admits an invariant trace.
\end{enumerate}
\end{corollary}

\begin{proof}
The implication (1) implies (2) is well known and follows immediately from the stable finiteness of $\mathcal{Q}_\omega$.  Assuming (2) holds, $A \rtimes_\lambda F$ admits a trace (Corollary 5.12 in \cite{Haagerup:Quasitraces}) which restricts to an invariant trace on $A$.  If (3) holds, let $\tau$ be an invariant trace on $A$.  By minimality, $\tau$ is faithful and hence induces a faithful trace on $\tau \circ \mathbb{E}$ on $A \rtimes_\lambda F$.  The trace $\tau \circ \mathbb{E}$ is MF and hence $A \rtimes_\lambda F$ is MF.
\end{proof}

Recall that a C$^*$-algebra $A$ has the \emph{ideal property} if every ideal of $A$ is generated as an ideal by the projections it contains.

\begin{corollary}\label{cor:AHCrossedProducts}
Suppose $A$ is an AH-algebra with the ideal property such that $\mathrm{K}_1(A)$ is a torsion group.  If $F$ is a free group acting on $A$, then every trace on $A \rtimes_\lambda F$ is MF.

If, moreover, $A$ is unital and the action is minimal, then the following are equivalent:
\begin{enumerate}
  \item $A \rtimes_\lambda F$ is MF;
  \item $A \rtimes_\lambda F$ is stably finite;
  \item $A$ admits an invariant trace.
\end{enumerate}
\end{corollary}

\begin{proof}
As $\mathrm{K}_*(A \otimes \mathcal{Q}) \cong \mathrm{K}_*(A) \otimes \mathbb{Q}$ is torsion free, $A \otimes \mathcal{Q}$ is an A$\mathbb{T}$-algebra by the main result of \cite{GongJiangLiPaniscu:AHimpliesAT}.  As $\mathrm{K}_1(A)$ is a torsion group, we have $\mathrm{K}_1(A \otimes \mathcal{Q}) = 0$.  Theorem 2.5 in \cite{Thomsen:ATimpliesAI} implies $A \otimes \mathcal{Q}$ is an AI-algebra and, in particular, is an A$\mathcal{C}_0$-algebra.  The result follows from Theorem \ref{thm:ANCCW1CrossedProducts} and Corollary \ref{cor:ANCCW1CrossedProducts}.
\end{proof}

\begin{corollary}\label{cor:SimpleNuclearCrossedProducts}
Suppose $A$ is separable, simple, unital C*-algebra satisfying the UCT such that $A \otimes \mathcal{Q}$ has finite nuclear dimension and suppose $F$ is a free group acting on $A$.  Assume either that the projection in $A$ separate traces on $A$ or that $\mathrm{K}_1(A)$ is a torsion group.

The following are equivalent.
\begin{enumerate}
  \item $A \rtimes_\lambda F$ is MF;
  \item $A \rtimes_\lambda F$ is stably finite;
  \item $A$ admits an invariant trace.
\end{enumerate}
Moreover, every trace on $A \rtimes_\lambda F$ is MF.
\end{corollary}

\begin{proof}
If the projections in $A$ separate traces on $A$, this Theorem 1.2 in \cite{RainoneSchafhauser}.  If $\mathrm{K}_1(A)$ is a torsion group, then $A \otimes \mathcal{Q}$ is an A$\mathcal{C}_0$-algebra by Corollary 13.47 in \cite{GongLinNiu:NCCW1} and Corollary D in \cite{TikuisisWhiteWinter}.
\end{proof}

\begin{remark}
The assumption that $A$ satisfies the UCT and that $A \otimes \mathcal{Q}$ has finite nuclear dimension in the above corollary may both be automatic when $A$ is nuclear.  For the finite nuclear dimension assumption, this is known to be the case whenever the set of extremal traces on $A$ is weak*-compact (see \cite{BosaBrownSatoTikuisisWhiteWinter}).  There is a good chance the conditions on the Elliott invariant in Corollary \ref{cor:SimpleNuclearCrossedProducts} are not necessary.
\end{remark}

We end this paper with a few remarks on potential further applications of these techniques.  If one could produce a classification theorem for A$\mathbb{T}$-algebras or A$\mathcal{C}$-algebras (and morphisms between them), one could probably reduce to the setting of AI-algebras or A$\mathcal{C}_0$-algebras, respectively.  An argument of this form was used in \cite{RainoneSchafhauser} as a way of reducing the MF-crossed product problem from real rank zero A$\mathbb{T}$-algebras to AF-algebras.

There is a good chance that if $A$ is an A$\mathcal{C}_0$-algebra and $F$ is a free group acting on $A$ such that $A \rtimes_\lambda F$ is stably finite, then $A \rtimes_\lambda F$ is MF.  The arguments would carry through in this setting if one could produce a faithful $\mathrm{Cu}$-morphism $\mathrm{Cu}(A \rtimes_\lambda F) \rightarrow \mathrm{Cu}(\mathcal{Q}_\omega)$ whenever the crossed product is stably finite.  In particular, it would be enough to know the following.

\begin{conjecture}
If $A$ is a separable, stably finite C*-algebra, then there is a faithful $\mathrm{Cu}$-morphism $\mathrm{Cu}(A) \rightarrow \mathrm{Cu}(\mathcal{Q}_\omega)$.
\end{conjecture}

Note that if $A$ is MF, then a faithful *-homomorphism $A \rightarrow \mathcal{Q}_\omega$ induces a faithful $\mathrm{Cu}$-morphism $\mathrm{Cu}(A) \rightarrow \mathrm{Cu}(\mathcal{Q}_\omega)$.  Hence a negative answer to the above conjecture would provide a counter example to the Blackadar-Kirchberg problem.

An interesting special case is $A = C_0(0, 1] \otimes \mathcal{O}_2$.  The Cuntz semigroup $\mathrm{Cu}(A)$ consists of open subsets of $(0, 1]$ with addition given by union and order given by inclusion.  Voiculescu's homotopy invariance of quasidiagonality shows that $A$ is MF and hence $A$ satisfies the conjecture.  It seems difficult to prove the existence of a faithful $\mathrm{Cu}$-morphism using only Cuntz semigroup techniques, however.

\end{document}